\documentclass[11pt,letterpaper,reqno]{amsart}

\usepackage{amssymb,amsmath,amsthm,amsfonts}
\usepackage{enumitem}
\usepackage{booktabs}
\usepackage{graphicx}
\usepackage[T1]{fontenc}
\usepackage{doi}
\usepackage{float}

\usepackage{tikz}
\usetikzlibrary{positioning, shapes.geometric, arrows.meta}
\usepackage{pgfplots}
\pgfplotsset{compat=1.18}

\usepackage{bookmark}
\usepackage{hyperref}
\hypersetup{pdfstartview={FitH}}

\newtheorem{thm}{Theorem}[section]
\newtheorem{lem}[thm]{Lemma}
\newtheorem{prop}[thm]{Proposition}

\theoremstyle{definition}

\newtheorem{rem}[thm]{Remark}

\numberwithin{equation}{section}

\newcommand{\M}{M}
\newcommand{\one}{1}
\newcommand{\Stau}{S(\tau)}
\newcommand{\Etau}{E(\tau)}
\newcommand{\norm}[1]{\left\lVert #1\right\rVert}
\newcommand{\abs}[1]{\left\lvert #1\right\rvert}
\newcommand{\mn}{\mathbb{M}_n}
\newcommand{\mnp}{\mathbb{M}_n^{+}}
\newcommand{\precprec}{\prec\!\prec}
\newcommand{\uin}[1]{%
	\left|\mkern-1.5mu\left|\mkern-1.5mu\left|#1\right|
	\mkern-1.5mu\right|\mkern-1.5mu\right|%
}
\newcommand{\mut}[2]{\mu(#1;\,#2)}
\newcommand{\dsf}[2]{d(#1;\,#2)}

\DeclareMathOperator{\diag}{diag}

\title[Bourin-type inequalities in symmetric spaces]
{Bourin-type inequalities for $\tau$-measurable operators in fully symmetric spaces}

\author[T.~Zhang]{Teng Zhang}
\address{School of Mathematics and Statistics, Xi'an Jiaotong University, Xi'an 710049, P. R. China}
\email{teng.zhang@stu.xjtu.edu.cn}


\begin{document}
	
	\begin{abstract}
		Let $\M\subset B(\mathcal H)$ be a semifinite von Neumann algebra, where $B(\mathcal H)$ denotes the algebra of all bounded linear operators on a Hilbert space $\mathcal H$, and let $\tau$ be a fixed faithful normal semifinite trace on $\M$.
		Let $\Etau$ be the fully symmetric space associated with a fully symmetric Banach function space $E$ on $[0,\infty)$.
		Using a complex interpolation argument based on the three-lines theorem on a strip, we show that for positive operators $a,b$ in $\Etau$ and $t\in[0,1]$, 
		\[
		\|a^t b^{1-t}+b^t a^{1-t}\|_{\Etau}\le 2^{\max\{2|t-1/2|-1/2,\;0\}}\;\|a+b\|_{\Etau},
		\]
		and in particular obtain the sharp constant $1$ for $t\in[1/4, 3/4]$:
		\[
		\|a^t b^{1-t}+b^t a^{1-t}\|_{\Etau}\le \|a+b\|_{\Etau}.
		\]
		This extends the work of Kittaneh--Ricard in [Linear Algebra Appl. \textbf{710} (2025), 356--362] and covers the results of Liu--He--Zhao in [Acta Math. Sci. Ser. B (Engl. Ed.) \textbf{46} (2026), 62--68].
	\end{abstract}
	
	\maketitle
	
	\section{Introduction}
	
	Let $\mn$ be the algebra of complex $n\times n$ matrices and let $\mnp$ be its cone of positive
	semidefinite matrices. A norm $\uin{\cdot}$ on $\mn$ is \emph{unitarily invariant} if
	$\uin{UAV}=\uin{A}$ for all $A\in\mn$ and all unitaries $U,V\in\mn$.
	
	The classical Heinz mean
	\[
	h_t(A,B)=A^tB^{1-t}+A^{1-t}B^t,\qquad A,B\in\mnp,\;t\in[0,1],
	\]
	satisfies the Heinz inequality (see e.g.\ \cite{Bha97,Kos98})
	\begin{equation}\label{eq:Heinz-matrix}
		\uin{h_t(A,B)}\le \uin{A+B}.
	\end{equation}
	Bhatia \cite[p.~265]{Bha97} also recorded the strengthened form with an extra variable $X\in\mn$ inserted between $A$
	and $B$:
	\begin{equation}\label{eq:Heinz-X-matrix}
		\uin{A^tXB^{1-t}+A^{1-t}XB^t}\le \uin{AX+XB},\qquad A,B\in\mnp,\;t\in[0,1].
	\end{equation}
	
Motivated by \eqref{eq:Heinz-matrix}, Bourin \cite{Bou09} asked whether the \emph{crossed} Heinz expression
\[
b_t(A,B):=A^tB^{1-t}+B^tA^{1-t}
\]
satisfies the same norm bound $\uin{b_t(A,B)}\le \uin{A+B}$ for all unitarily invariant norms and all
$t\in[0,1]$. Subsequent developments provided affirmative answers in various parameter ranges and for several
related inequalities; see \cite{HHK21,HHK22,HHK23} and the references therein. More recently, Kittaneh and
Ricard \cite{KR25} confirmed the conjectured bound for $t\in[1/4,3/4]$ and, moreover, proved that for all
$t\in[0,1]$,
\begin{equation}\label{eq:KR-bound}
\uin{A^tB^{1-t}+B^tA^{1-t}}
	\le
	2^{\max\{\,2\abs{t-1/2}-1/2,\;0\,\}}
\uin{A+B}.
\end{equation}

	In this paper, we study these inequalities in the general framework of fully symmetric spaces of
	$\tau$-measurable operators.
	Throughout, $\M\subset B(\mathcal H)$ denotes a semifinite von Neumann algebra with a fixed faithful normal semifinite trace $\tau$, where $B(\mathcal H)$ denotes the algebra of all bounded linear operators on a Hilbert space $\mathcal H$.
	Let $\Stau$ be the $*$-algebra of all $\tau$-measurable operators affiliated with $\M$, see \cite{DDSZ20,FK86}. For $x\ge 0$ in $\Stau$, we mean that $x$ is a positive operator in $\Stau$.
	Given a fully symmetric Banach function space $E$ on $[0,\infty)$, we write $\Etau$ for the associated fully symmetric space
	defined via generalized singular numbers.
	We assume that $E$ (equivalently, $\Etau$) has the Fatou property. 
	
Dodds, Dodds, Sukochev and Zanin \cite{DDSZ20} established the following Heinz-type inequality for $\tau$-measurable operators, which may be viewed as the $\tau$-measurable analogue of \eqref{eq:Heinz-X-matrix}.

	\begin{thm}[Heinz-type inequality in $\Etau$, {\cite[Corollary~5.8(i)]{DDSZ20}}]\label{thm:heinz-Etau}
		Let $a,b\ge 0$ in $\Stau$ and $z\in \Stau$. If $az,zb\in \Etau$, then for every $t\in[0,1]$,
		\[
		\norm{a^t z b^{1-t}+a^{1-t} z b^t}_{\Etau}\le \norm{az+zb}_{\Etau}.
		\]
	\end{thm}
	
	In particular, if $a,b\in \Etau$, then taking $z=\one$ in Theorem~\ref{thm:heinz-Etau} yields the usual Heinz inequality in $\Etau$:
	\[
	\norm{a^tb^{1-t}+a^{1-t}b^t}_{\Etau}\le \norm{a+b}_{\Etau}, \qquad t\in [0,1].
	\]
	
	In this paper, we focus on Bourin's crossed Heinz expression for $\tau$-measurable operators in fully symmetric spaces.
	Liu, He and Zhao \cite{LHZ26}  proved the following bound for $t\in[0,1/4]\cup[3/4,1]$.
	\begin{thm}\label{thm:Liu_He_Zhao}
		Let $a,b\ge 0$ in $\Etau$ and $t\in[0,1/4]\cup [3/4,1]$. Then
		\begin{equation*}
			\norm{a^t b^{1-t}+b^t a^{1-t}}_{\Etau}
			\le
			2^{2\abs{t-1/2}-1/2}\;\norm{a+b}_{\Etau}.
		\end{equation*}
	\end{thm}
	By using a complex interpolation argument based on the three-lines theorem on a strip, we prove that the optimal constant $1$ holds on the full central range $t\in[1/4,3/4]$ and obtain an explicit bound valid for all $t\in[0,1]$, which generalizes  Kittaneh--Ricard's result~\eqref{eq:KR-bound} and Theorem~\ref{thm:Liu_He_Zhao}.
	\begin{thm}[Bourin-type bound in $\Etau$]\label{thm:main}
		Let $a,b\ge 0$ in $\Etau$ and $t\in[0,1]$. Then
		\begin{equation}\label{eq:main}
			\norm{a^t b^{1-t}+b^t a^{1-t}}_{\Etau}
			\le
			2^{\max\{\,2\abs{t-1/2}-1/2,\;0\,\}}\;\norm{a+b}_{\Etau}.
		\end{equation}
		In particular, if $t\in[1/4,3/4]$, then
		\[
		\norm{a^t b^{1-t}+b^t a^{1-t}}_{\Etau}\le \norm{a+b}_{\Etau}.
		\]
	\end{thm}
	
	\medskip
	
	\noindent\textbf{Organization of this paper.}
	Section~\ref{sec:prelim} recalls the notation for $\tau$-measurable operators and fully symmetric spaces $\Etau$.
	It also collects several tools used later, including a H\"older-type inequality, the three-lines lemma on a strip,
	an arithmetic--geometric mean inequality, an Araki--Lieb--Thirring-type (ALT-type) submajorization statement, and a Douglas-type factorization lemma.
	Section~\ref{sec:bourin} establishes a complex interpolation estimate for an auxiliary expression $f_t$,
	then proves a bridge inequality via a $2\times2$ block-positivity argument, and finally combines these ingredients to prove
	Theorem~\ref{thm:main}.
	
	\section{Preliminaries}\label{sec:prelim}
	
	\subsection{$\tau$-measurable operators and generalized singular numbers}
	
	We follow \cite[\S2]{DDSZ20}.
	Let $\M\subset B(\mathcal H)$ be a semifinite von Neumann algebra with faithful normal semifinite trace $\tau$.
	The identity in $\M$ is denoted by $\one$, and $P(\M)$ denotes the lattice of projections in $\M$.
	
	A closed densely defined operator $x$ on $\mathcal H$ is \emph{affiliated} with $\M$ if it commutes with the commutant $\M'$.
	A closed densely defined operator $x$ affiliated with $\M$ is called \emph{$\tau$-measurable} if there exists $s\ge 0$ such that
	$\tau(e^{|x|}(s,\infty))<\infty$, where $e^{|x|}$ is the spectral measure of $|x|$.
	The collection of all $\tau$-measurable operators is denoted by $\Stau$; with sums/products taken as closures,
	$\Stau$ is a $*$-algebra.
	
	For $x\in\Stau$, define the distribution function
	\[
	\dsf{s}{|x|}:=\tau\bigl(e^{|x|}(s,\infty)\bigr),\qquad s\ge 0,
	\]
	and the generalized singular value function
	\[
	\mut{t}{x}:=\inf\{\,s\ge 0:\ \dsf{s}{|x|}\le t\,\},\qquad t\ge 0.
	\]
	We write $\mu(x)=\mu(\cdot;\,x)=\mu(\cdot;\,|x|)$.
	
	We use the standard properties (see \cite{DDSZ20,FK86}):
	\begin{itemize}[leftmargin=2.2em]
		\item $\mu(uxv)\le \|u\|_\infty\,\|v\|_\infty\,\mu(x)$ for all $u,v\in \M$ and $x\in\Stau$;
		\item if $0\le x\in\Stau$ and $f$ is increasing continuous on $[0,\infty)$ with $f(0)\ge 0$, then $\mu(f(x))=f(\mu(x))$.
	\end{itemize}
	
	For $0\le x\in\Stau$, the support projection is
	\[
	s(x):=\one-e^{x}(\{0\})=\one_{(0,\infty)}(x)\in P(\M).
	\]
	We adopt the convention $x^0:=s(x)$.
	
	\subsection{Fully symmetric spaces $\Etau$}
	
	Let $E$ be a fully symmetric Banach function space on $[0,\infty)$ (Lebesgue measure).
	The associated fully symmetric space $\Etau$ is defined by
	\[
	\Etau:=\{x\in\Stau:\ \mu(x)\in E\},\qquad
	\|x\|_{\Etau}:=\|\mu(x)\|_{E}.
	\]
	Then $\Etau$ is an $\M$-bimodule: for $a,b\in \M$ and $x\in\Etau$,
	\begin{equation}\label{eq:bimodule}
		\norm{axb}_{\Etau}\le \norm{a}_{\infty}\,\norm{x}_{\Etau}\,\norm{b}_{\infty}.
	\end{equation}
	In particular, if $a,b\in\M$, then
	\begin{equation}\label{eq:mu-axb}
		\mut{t}{axb}\le \norm{a}_{\infty}\,\norm{b}_{\infty}\,\mut{t}{x}\qquad (t>0).
	\end{equation}
	
	We also use Hardy--Littlewood submajorization in the sense of \cite[\S2]{DDSZ20}:
	for $x\in S(\tau)$ and $y\in S(\sigma)$ (possibly in different algebras),
	\[
	x\precprec y \quad \Longleftrightarrow\quad \int_0^t \mut{s}{x}\,ds\le \int_0^t \mut{s}{y}\,ds\ \ \text{for all }t\ge 0.
	\]
	Since $\Etau$ is fully symmetric, $0\le u\precprec v$ and $v\in\Etau$ imply $u\in\Etau$ and $\|u\|_{\Etau}\le \|v\|_{\Etau}$.
	
	\begin{lem}[Unitary invariance]\label{lem:unitary-inv}
		If $u,v\in\M$ are unitaries and $x\in\Etau$, then $\mu(uxv)=\mu(x)$ and hence $\|uxv\|_{\Etau}=\|x\|_{\Etau}$.
	\end{lem}
	
	\subsection{Spectral truncations}
	
	For $0\le x\in\Stau$ and $n>0$ we define the usual truncation
	\[
	x\wedge n := \min\{x,n\one\}=x\,\one_{[0,n]}(x)+n\,\one_{(n,\infty)}(x),
	\]
	defined by Borel functional calculus. Then $x\wedge n\ge 0$ and $x\wedge n\to x$ in measure as $n\to\infty$.
	
	\begin{lem}[Truncations commute with positive powers]\label{lem:trunc-powers}
		Let $0\le x\in \Stau$ and $n>0$. Then for every $\alpha>0$,
		\[
		(x\wedge n)^{\alpha}=x^{\alpha}\wedge n^{\alpha}.
		\]
		In particular, $(x\wedge n)^{\alpha}\to x^{\alpha}$ in measure as $n\to\infty$.
	\end{lem}
	\begin{proof}
		This is immediate from scalar functional calculus applied to $\lambda\mapsto \min\{\lambda,n\}$ and then to $\lambda\mapsto \lambda^\alpha$.
	\end{proof}
	
	\subsection{Measure topology and continuity tools}
	
	Following \cite[\S2]{DDSZ20}, for $\varepsilon,\delta>0$ let $V(\varepsilon,\delta)$ be the set of all
	$x\in \Stau$ for which there exists a projection $p\in P(\M)$ such that $p(\mathcal H)\subset D(x)$,
	$\|xp\|_\infty\le \varepsilon$, and $\tau(\one-p)\le \delta$.
	Then $\{V(\varepsilon,\delta):\varepsilon,\delta>0\}$ is a base at $0$ for a metrizable Hausdorff topology on $\Stau$,
	called the \emph{measure topology}; equipped with this topology, $\Stau$ is a complete topological $*$-algebra.
	
	We will repeatedly use the standard facts (see \cite[\S2]{DDSZ20}): a sequence $x_n\to 0$ in measure iff
	$\mut{t}{x_n}\to 0$ for all $t>0$, and if $a_n,a\in\Stau$ are self-adjoint and $a_n\to a$ in measure, then
	$f(a_n)\to f(a)$ in measure for every continuous $f:\mathbb R\to\mathbb C$ (Tikhonov).
	
	\begin{lem}[Measure-topology continuity]\label{lem:measure-cont}
		Let $(x_n)$ and $(y_n)$ be sequences in $\Stau$ such that $x_n\to x$ and $y_n\to y$ in measure. Then:
		\begin{enumerate}[label=\textup{(\roman*)},leftmargin=2.2em]
			\item $x_n^*\to x^*$ in measure;
			\item $x_n y_n \to xy$ in measure (products taken as closures in $\Stau$);
			\item if in addition each $x_n$ is self-adjoint and $f:\mathbb{R}\to\mathbb{C}$ is continuous, then
			$f(x_n)\to f(x)$ in measure.
		\end{enumerate}
	\end{lem}
	
	Lemma~\ref{lem:measure-cont} is standard since $\Stau$ is a complete topological $*$-algebra for the measure topology; (iii) follows from Tikhonov's theorem, see \cite[\S2]{DDSZ20}.
	
	\begin{rem}\label{rem:finite-corner-in-Etau}
		If $e\in \M$ is a projection with $\tau(e)<\infty$, then $e\in \Etau$.
		Indeed, $\mu(e)=\one_{[0,\tau(e))}$, and $\one_{[0,t)}\in E$ for every $t<\infty$ in a Banach function space.
		Consequently, every bounded operator in the finite von Neumann algebra $e\M e$ belongs to $\Etau$.
	\end{rem}
	
	\begin{lem}[Fatou-type lower semicontinuity, {\cite[Proposition~6.3]{Pa08}}]\label{lem:fatou-lsc}
		Assume that $\Etau$ has the Fatou property.
		Let $(x_n)\subset \Etau$ be bounded in $\|\cdot\|_{\Etau}$ and converging to $x\in \Stau$ in measure.
		Then $x\in \Etau$ and
		\[
		\|x\|_{\Etau}\le \liminf_{n\to\infty}\|x_n\|_{\Etau}.
		\]
	\end{lem}
	
	\begin{lem}[Vanishing at infinity when $\one\notin E$]\label{lem:vanishing-infty}
		Assume that $\tau(\one)=\infty$ and $\one\notin E$.
		Then for every $f\in E$ we have $\lim_{t\to\infty} f^*(t)=0$, where $f^*$ is the decreasing rearrangement.
		In particular, if $x\in \Etau$, then $\mut{t}{x}\to 0$ as $t\to\infty$, and equivalently
		\[
		\tau\bigl(\one_{(\delta,\infty)}(|x|)\bigr)<\infty\qquad(\delta>0).
		\]
	\end{lem}
	\begin{proof}
		This is standard for Banach function spaces on $[0,\infty)$: if $f^*$ does not vanish at infinity,
		then boundedness of $\|f\|_E$ forces $\one\in E$, a contradiction. Applying this to $f=\mu(x)\in E$
		yields $\mut{t}{x}\to 0$.
	\end{proof}
	
	\begin{lem}[Sequential finite-corner approximation for $\tau$-compact positives]\label{lem:seq-cutdown}
		Assume $\tau(\one)=\infty$ and $\one\notin E$, so that $a,b\in\Etau_+$ satisfy
		$\tau(\one_{(\delta,\infty)}(a))<\infty$ and $\tau(\one_{(\delta,\infty)}(b))<\infty$ for all $\delta>0$.
		For $k\in\mathbb N$ set
		\[
		p_k:=\one_{(1/k,\infty)}(a),\qquad q_k:=\one_{(1/k,\infty)}(b),\qquad e_k:=p_k\vee q_k.
		\]
		Then $\tau(e_k)<\infty$, and
		\[
		a-e_kae_k\in\M,\quad \|a-e_kae_k\|_\infty\le \frac{2}{k},
		\qquad
		b-e_kbe_k\in\M,\quad \|b-e_kbe_k\|_\infty\le \frac{2}{k}.
		\]
		In particular, $e_kae_k\to a$ and $e_kbe_k\to b$ in measure as $k\to\infty$.
	\end{lem}
	\begin{proof}
		Since $\tau(p_k)<\infty$ and $\tau(q_k)<\infty$, we have $\tau(e_k)\le \tau(p_k)+\tau(q_k)<\infty$.
		Let $\delta:=1/k$ and $p:=p_k=\one_{(\delta,\infty)}(a)\le e_k$. Because $p$ is a spectral projection of $a$,
		it commutes with $a$ and $a(1-p)\le \delta(1-p)$, hence $\|a(1-p)\|_\infty\le \delta$.
		For $e_k\ge p$ we compute
		\[
		a-e_kae_k = a(1-p)-e_ka(1-p)e_k,
		\]
		and both terms are bounded by $\delta$ in operator norm, so $\|a-e_kae_k\|_\infty\le 2\delta=2/k$.
		The estimate for $b$ is identical. Operator-norm convergence implies measure convergence.
	\end{proof}
	
	\subsection{Three-lines lemma on a strip}
	
	Let $\Delta=\{z\in\mathbb C:\ 0<\Re z<1\}$ and $\overline{\Delta}=\{z\in\mathbb C:\ 0\le \Re z\le 1\}$. The well-known three--lines theorem can be stated as follows.
	\begin{lem}[Three-lines lemma, \cite{BL76}]\label{lem:three-lines}
		Let $X$ be a Banach space and let $f:\overline{\Delta}\to X$ be bounded and continuous on
		$\overline{\Delta}$ and holomorphic on $\Delta$. For $u\in[0,1]$ set
		\[
		M(u):=\sup_{s\in\mathbb R}\|f(u+is)\|_X.
		\]
		Then $\log M(u)$ is convex on $[0,1]$. In particular, if $0\le \alpha<t<\beta\le 1$ and
		$t=(1-\theta)\alpha+\theta\beta$, then
		\[
		M(t)\le M(\alpha)^{1-\theta}M(\beta)^{\theta}.
		\]
	\end{lem}
	 For the scalar-valued case (the classical three-lines lemma in a strip), see
	\cite[p.~133, Problem~3]{SS03}.
	For the Banach-valued extension, fix $\varphi\in X^*$ and apply the scalar three-lines lemma
	to the bounded holomorphic scalar function $\varphi\circ f$ on $\Delta$.
	Taking the supremum over all $\varphi\in X^*$ with $\|\varphi\|\le 1$ yields the stated bound for $M(u)$,
	and hence the convexity of $\log M(u)$.
	\subsection{Auxiliary inequalities}
	
	\begin{lem}[Complex powers for bounded invertibles, {\cite[Lemma~4.1]{DDSZ20}}]\label{lem:complex-powers}
		Let $0\le a\in \M$ be invertible. Then the map $z\mapsto a^z$ is analytic on $\mathbb C$ with values in $\M$.
	\end{lem}
	
	\begin{lem}[H\"older-type inequality, {\cite[p.~17]{DDZ20}}]\label{lem:holder}
		Let $x,y\in \Stau$ and let $m,n>1$ satisfy $\frac1m+\frac1n=1$.
		If $|x|^m\in \Etau$ and $|y|^n\in \Etau$, then $xy\in \Etau$ and
		\[
		\|xy\|_{\Etau}\le \||x|^m\|_{\Etau}^{1/m}\ \||y|^n\|_{\Etau}^{1/n}.
		\]
	\end{lem}
	
	\begin{lem}[Arithmetic--geometric mean inequality, {\cite[Corollary~4.8]{DDSZ20}}]\label{lem:AMGM}
		Let $a,b\ge 0$ in $\Etau$. Then $a^{1/2}b^{1/2}\in\Etau$ and
		\begin{equation}\label{eq:AMGM}
			2\|a^{1/2}b^{1/2}\|_{\Etau}\le \|a+b\|_{\Etau}.
		\end{equation}
	\end{lem}
	
	\begin{lem}[ALT-type submajorization, \cite{Han16}]\label{lem:alt-submaj}
		Let $a,b\ge 0$ in $\Stau$ and let $r\ge 1$. Then
		\[
		|ab|^{\,r}\ \precprec\ a^{r}b^{r}.
		\]
		Here $a^{r}b^{r}$ need not be positive; the relation $\precprec$ is understood in terms of generalized
		singular numbers, i.e.\ $\mu(|ab|^{\,r})\precprec \mu(a^{r}b^{r})$.
	\end{lem}

	\begin{lem}[Douglas-type factorization,{ \cite{Dou66} or \cite[Lemma~3.32, p.~29]{HK21}}]\label{lem:douglas}
		Let $ a,b\ge 0$ in $\Stau$ and $c\in \Stau$. Consider the $2\times2$ block operator
$
		\begin{pmatrix}
			a & c\\
			c^{*} & b
		\end{pmatrix}
		\in \mathbb M_2(\Stau)
$.
		Then
$
		\begin{pmatrix}
			a & c\\
			c^{*} & b
		\end{pmatrix}\ge 0
	$
		if and only if there exists a contraction $z\in \M$ such that	$c = a^{1/2}\, z\, b^{1/2}$.

	\end{lem}
	
	\section{Proof of Theorem~\ref{thm:main}}\label{sec:bourin}
	
	Throughout this section, fix $a,b\ge 0$ in $\Etau$. For $t\in[0,1]$ set
	\[
	b_t=b_t(a,b):=a^t b^{1-t}+b^t a^{1-t}.
	\]
	For $t\in[\tfrac12,1]$ define
	\[
	f_t=f_t(a,b):=b^{1-t}a^{2t-1}b^{1-t}+a^{1-t}b^{2t-1}a^{1-t}.
	\]
	We use the convention $u^0=s(u)$ for $u\ge 0$.
	
	\subsection{A complex interpolation estimate for $f_t$}
	
	\begin{thm}\label{thm:f-bound}
		For every $t\in[\tfrac12,1]$,
		\[
		\norm{f_t}_{\Etau}\le 2^{\max\{4t-3,\ 0\}}\ \norm{a+b}_{\Etau}.
		\]
	\end{thm}
	
	\begin{proof}
		\textbf{Step A: the endpoint $t=\tfrac12$.}
		When $t=\tfrac12$ we have
		\[
		f_{1/2}=b^{1/2}s(a)b^{1/2}+a^{1/2}s(b)a^{1/2}\le b+a,
		\]
		since $s(a),s(b)\le \one$ and $0\le c^{1/2}pc^{1/2}\le c$ for any projection $p$.
		By solidity, $\|f_{1/2}\|_{\Etau}\le \|a+b\|_{\Etau}$.
		
		\medskip
		\textbf{Step B: the endpoint $t=1$.}
		Here $f_1=s(b)as(b)+s(a)bs(a)$. Since $s(a),s(b)\in \M$ are projections, by \eqref{eq:bimodule},
		\[
		\|s(b)as(b)\|_{\Etau}\le \|a\|_{\Etau},\qquad \|s(a)bs(a)\|_{\Etau}\le \|b\|_{\Etau}.
		\]
		Since $0\le a\le a+b$ and $0\le b\le a+b$, solidity of $\Etau$ yields $\|a\|_{\Etau},\|b\|_{\Etau}\le \|a+b\|_{\Etau}$.
		Hence
		\[
		\|f_1\|_{\Etau}\le \|a\|_{\Etau}+\|b\|_{\Etau}\le 2\|a+b\|_{\Etau},
		\]
		matching $2^{4\cdot1-3}=2$.
		
		\medskip
		For the remainder, assume $t\in(\tfrac12,1)$ so that $1-t>0$ and $2t-1>0$.
		
		\medskip
		\textbf{Step 0: reduction to bounded invertibles in a unital finite setting.}
		We first establish the estimate for bounded invertible positive operators in a von Neumann algebra with a unit $e$ such that
		$e\M e\subset \Etau$ (e.g.\ $e=\one$ when $\tau(\one)<\infty$ or $\one\in E$, or $e=e_k$ a finite projection).
		Then we pass to general $a,b\in\Etau$ by truncation and (when $\tau(\one)=\infty$ and $\one\notin E$) by a finite-corner approximation.
		
		\smallskip
		\emph{Case 1: $\tau(\one)<\infty$ or $\one\in E$ (so $\M\subset \Etau$).}
		For $n\in\mathbb N$ set $a_n:=a\wedge n$, $b_n:=b\wedge n$, and for $\varepsilon>0$ define
		\[
		a_{n,\varepsilon}:=a_n+\varepsilon\one,\qquad b_{n,\varepsilon}:=b_n+\varepsilon\one.
		\]
		Note that $0\le a_n+b_n\le a+b$, hence $\|a_n+b_n\|_{\Etau}\le \|a+b\|_{\Etau}$, and
		\[
		\|a_{n,\varepsilon}+b_{n,\varepsilon}\|_{\Etau}
		\le \|a+b\|_{\Etau}+2\varepsilon\|\one\|_{\Etau}
		\quad(\text{when }\one\in\Etau).
		\]
		
		\smallskip
		\emph{Case 2: $\tau(\one)=\infty$ and $\one\notin E$.}
		Then Lemma~\ref{lem:vanishing-infty} implies $\mut{t}{a},\mut{t}{b}\to 0$ as $t\to\infty$, so Lemma~\ref{lem:seq-cutdown} applies.
		Let $(e_k)$ be as in Lemma~\ref{lem:seq-cutdown}. For $k,n\in\mathbb N$ set
		\[
		a_{k,n}:=e_k\,(a\wedge n)\,e_k,\qquad b_{k,n}:=e_k\,(b\wedge n)\,e_k,
		\]
		and for $\varepsilon>0$ define
		\[
		a_{k,n,\varepsilon}:=a_{k,n}+\varepsilon e_k,\qquad
		b_{k,n,\varepsilon}:=b_{k,n}+\varepsilon e_k,
		\]
		which are bounded and invertible in the finite von Neumann algebra $e_k\M e_k$ (unit $e_k$).
		By Remark~\ref{rem:finite-corner-in-Etau}, $e_k\M e_k\subset \Etau$.
		
		In either case, it suffices to prove the estimate for bounded invertible positive $a,b$ in a von Neumann algebra with unit $e$ such that $e\M e\subset \Etau$.
		For the rest of the proof we temporarily write $a,b$ for such a bounded invertible pair.
		
		\smallskip
		\emph{Remark on unitaries in corners.}
		If $e\neq \one$, then every unitary $u\in e\M e$ may be viewed as a unitary in $\M$ via the standard extension
		$U:=u+(\one-e)\in\M$.
		Since $a=eae$ and $b=ebe$, conjugation by $U$ agrees with conjugation by $u$ on expressions involving $a$ and $b$.
		We will freely use Lemma~\ref{lem:unitary-inv} for such unitaries (in particular, for imaginary powers $a^{is}$, $b^{is}$).
		
		\medskip
		\textbf{Step 1: complex interpolation on a strip.}
		Consider the strip
		\[
		\Sigma:=\{z\in\mathbb C:\ \tfrac12\le \Re z\le 1\}.
		\]
		For $z\in\Sigma$ define
		\[
		\Phi(z):=b^{1-z}a^{2z-1}b^{1-z}+a^{1-z}b^{2z-1}a^{1-z}.
		\]
		By Lemma~\ref{lem:complex-powers}, $z\mapsto a^z$ and $z\mapsto b^z$ are operator-norm analytic, hence $\Phi$ is analytic on the interior of $\Sigma$
		and continuous on $\Sigma$ as an $\Etau$-valued map.
		Indeed, since $e\M e\subset \Etau$ and $\|c\|_{\Etau}\le \|c\|_\infty\,\|e\|_{\Etau}$ for $c\in e\M e$,
		the inclusion $e\M e\hookrightarrow \Etau$ is continuous, so operator-norm analyticity implies $\Etau$-valued analyticity.
		For $u\in[\tfrac12,1]$, define
		\[
		M(u):=\sup_{s\in\mathbb R}\norm{\Phi(u+is)}_{\Etau}.
		\]
		Since $a,b$ are bounded and the imaginary powers are unitaries, $\|a^{\alpha+is}\|_\infty=\|a^\alpha\|_\infty$ and $\|b^{\alpha+is}\|_\infty=\|b^\alpha\|_\infty$.
		Hence $\Phi$ is bounded on $\Sigma$ by \eqref{eq:bimodule}, so $M(u)<\infty$ for all $u\in[\tfrac12,1]$.
		
		\medskip
		\textbf{Step 2: the boundary $\Re z=1$.}
		For $z=1+is$, we have
		\[
		\Phi(1+is)
		= b^{-is}a^{1+2is}b^{-is}+a^{-is}b^{1+2is}a^{-is}.
		\]
		Writing
		\[
		b^{-is}a^{1+2is}b^{-is}=(b^{-is}a^{is})\,a\,(a^{is}b^{-is}),
		\qquad
		a^{-is}b^{1+2is}a^{-is}=(a^{-is}b^{is})\,b\,(b^{is}a^{-is}),
		\]
		and noting the outer factors are unitaries, Lemma~\ref{lem:unitary-inv} gives
		\[
		\norm{\Phi(1+is)}_{\Etau}\le \norm{a}_{\Etau}+\norm{b}_{\Etau}\le 2\norm{a+b}_{\Etau}.
		\]
		Thus $M(1)\le 2\norm{a+b}_{\Etau}$.
		
		\medskip
		\textbf{Step 3: the line $\Re z=\tfrac34$.}
		For $z=\tfrac34+is$, write $u_s:=b^{is}$ and $v_s:=a^{is}$ (unitaries).
		A direct computation yields
		\[
		\Phi\Bigl(\tfrac34+is\Bigr)
		= u_{-s}\Bigl(b^{1/4}a^{1/2}v_{2s}b^{1/4}\Bigr)u_{-s}
		+ v_{-s}\Bigl(a^{1/4}b^{1/2}u_{2s}a^{1/4}\Bigr)v_{-s}.
		\]
		By Lemma~\ref{lem:unitary-inv} and the triangle inequality,
		\begin{equation}\label{eq:phi34}
			\norm{\Phi(\tfrac34+is)}_{\Etau}
			\le
			\norm{b^{1/4}a^{1/2}v_{2s}b^{1/4}}_{\Etau}
			+
			\norm{a^{1/4}b^{1/2}u_{2s}a^{1/4}}_{\Etau}.
		\end{equation}
		
		Consider the first term and factor
		\[
		b^{1/4}a^{1/2}v_{2s}b^{1/4}
		=\bigl(b^{1/4}a^{1/4}v_{s}\bigr)\,\bigl(v_sa^{1/4}b^{1/4}\bigr).
		\]
		Let $A:=b^{1/4}a^{1/4}v_s$ and $B:=v_sa^{1/4}b^{1/4}$. Then
		\[
		|A|^2=A^*A=v_s^*\,|b^{1/4}a^{1/4}|^2\,v_s,\qquad |B|^2=B^*B=|a^{1/4}b^{1/4}|^2.
		\]
		By Lemma~\ref{lem:holder} with $(m,n)=(2,2)$ and Lemma~\ref{lem:unitary-inv},
		\[
		\norm{b^{1/4}a^{1/2}v_{2s}b^{1/4}}_{\Etau}
		\le
		\norm{\abs{b^{1/4}a^{1/4}}^2}_{\Etau}^{1/2}\,
		\norm{\abs{a^{1/4}b^{1/4}}^2}_{\Etau}^{1/2}.
		\]
		The same estimate holds for the second term in \eqref{eq:phi34}. Hence
		\[
		\norm{\Phi(\tfrac34+is)}_{\Etau}
		\le
		2\,
		\norm{\abs{b^{1/4}a^{1/4}}^2}_{\Etau}^{1/2}\,
		\norm{\abs{a^{1/4}b^{1/4}}^2}_{\Etau}^{1/2}.
		\]
		By Lemma~\ref{lem:alt-submaj} with $r=2$,
		\[
		\abs{b^{1/4}a^{1/4}}^2\precprec b^{1/2}a^{1/2},
		\qquad
		\abs{a^{1/4}b^{1/4}}^2\precprec a^{1/2}b^{1/2}.
		\]
		By full symmetry and $\norm{u}_{\Etau}=\norm{u^*}_{\Etau}$,
		\[
		\norm{\abs{b^{1/4}a^{1/4}}^2}_{\Etau}\le \norm{a^{1/2}b^{1/2}}_{\Etau},
		\qquad
		\norm{\abs{a^{1/4}b^{1/4}}^2}_{\Etau}\le \norm{a^{1/2}b^{1/2}}_{\Etau}.
		\]
		Therefore
		\[
		\norm{\Phi(\tfrac34+is)}_{\Etau}
		\le 2\norm{a^{1/2}b^{1/2}}_{\Etau}
		\le \norm{a+b}_{\Etau},
		\]
		where the last step is \eqref{eq:AMGM}. Thus $M(\tfrac34)\le \norm{a+b}_{\Etau}$.
		
		\medskip
		\textbf{Step 4: the line $\Re z=\tfrac12$.}
		For $z=\tfrac12+is$ one can rewrite
		\[
		\Phi(\tfrac12+is)
		=
		a^{1/2}v_{-s}u_{2s}v_{-s}a^{1/2}
		+
		b^{1/2}u_{-s}v_{2s}u_{-s}b^{1/2}.
		\]
		Let $R=[\,a^{1/2}\ \ b^{1/2}\,]$ (a row operator) and define the diagonal unitary
		\[
		W_s:=\diag\bigl(v_{-s}u_{2s}v_{-s},\ u_{-s}v_{2s}u_{-s}\bigr)\in \mathbb M_2(\M).
		\]
		Then $\Phi(\tfrac12+is)=R\,W_s\,R^*$ and $a+b=R\,R^*$.
		As in the standard $2\times2$ trick, the block matrix
		\[
		\begin{pmatrix}
			a+b & \Phi(\tfrac12+is)\\
			\Phi(\tfrac12+is)^* & a+b
		\end{pmatrix}
		=
		\begin{pmatrix}
			R & 0\\
			0 & R
		\end{pmatrix}
		\begin{pmatrix}
			I & W_s\\
			W_s^* & I
		\end{pmatrix}
		\begin{pmatrix}
			R^* & 0\\
			0 & R^*
		\end{pmatrix}
		\]
		is positive, because $W_s$ is unitary and
		\[
		\begin{pmatrix}I&W_s\\W_s^*&I\end{pmatrix}
		=
		\begin{pmatrix}I\\W_s^*\end{pmatrix}
		\begin{pmatrix}I&W_s\end{pmatrix}\ge0.
		\]
		By Lemma~\ref{lem:douglas}, there exists a contraction $c_s\in \M$ such that
		\[
		\Phi(\tfrac12+is)=(a+b)^{1/2}c_s(a+b)^{1/2}.
		\]
		Since $c_s\in \M$ and $a+b\in\Etau$, the bimodule estimate \eqref{eq:bimodule} implies
		$(a+b)c_s$ and $c_s(a+b)$ belong to $\Etau$.
		Applying Theorem~\ref{thm:heinz-Etau} with $a=b=a+b$, $z=c_s$ and $t=\tfrac12$ gives
		\[
		2\,\norm{\Phi(\tfrac12+is)}_{\Etau}
		\le \norm{(a+b)c_s+c_s(a+b)}_{\Etau}.
		\]
		Using \eqref{eq:bimodule} and $\norm{c_s}_\infty\le 1$,
		\[
		\norm{(a+b)c_s+c_s(a+b)}_{\Etau}
		\le \norm{(a+b)c_s}_{\Etau}+\norm{c_s(a+b)}_{\Etau}
		\le 2\norm{a+b}_{\Etau}.
		\]
		Hence $M(\tfrac12)\le \norm{a+b}_{\Etau}$.
		
		\medskip
		\textbf{Step 5: interpolate.}
		Apply Lemma~\ref{lem:three-lines} to the reparametrization $g(w):=\Phi(\tfrac12+\tfrac14 w)$ on $0\le \Re w\le 1$.
		By Steps~3--4, $M(\tfrac12)\le \|a+b\|_{\Etau}$ and $M(\tfrac34)\le \|a+b\|_{\Etau}$, hence
		$M(t)\le \|a+b\|_{\Etau}$ for all $t\in[\tfrac12,\tfrac34]$.
		
		For $t\in[\tfrac34,1]$, apply Lemma~\ref{lem:three-lines} to $h(w):=\Phi(\tfrac34+\tfrac14 w)$ on $0\le \Re w\le 1$.
		Then for $t=(1-\theta)\tfrac34+\theta\cdot 1$ (so $\theta=4t-3$),
		\[
		\|\Phi(t)\|_{\Etau}\le M(t)\le \bigl(M(\tfrac34)\bigr)^{1-\theta}\bigl(M(1)\bigr)^\theta
		\le \|a+b\|_{\Etau}^{1-\theta}(2\|a+b\|_{\Etau})^\theta
		=2^{\,4t-3}\|a+b\|_{\Etau}.
		\]
		Since $f_t=\Phi(t)$, we have proved the desired bound in the bounded invertible setting.
		
		\medskip
		\textbf{Step 6: pass to general $a,b\in\Etau$.}
		
		\smallskip
		\emph{Case 1: $\tau(\one)<\infty$ or $\one\in E$.}
		Fix $n$ and apply the bound to $a_{n,\varepsilon},b_{n,\varepsilon}$:
		\[
		\|f_t(a_{n,\varepsilon},b_{n,\varepsilon})\|_{\Etau}
		\le 2^{\max\{4t-3,0\}}\ \|a_{n,\varepsilon}+b_{n,\varepsilon}\|_{\Etau}.
		\]
		As $\varepsilon\downarrow0$, functional calculus is norm-continuous on bounded sets (since $1-t,2t-1>0$),
		hence $f_t(a_{n,\varepsilon},b_{n,\varepsilon})\to f_t(a_n,b_n)$ in operator norm, and therefore in measure.
		By Lemma~\ref{lem:fatou-lsc},
		\[
		\|f_t(a_n,b_n)\|_{\Etau}\le 2^{\max\{4t-3,0\}}\ \|a+b\|_{\Etau}.
		\]
		Now let $n\to\infty$. Since $a_n\to a$ and $b_n\to b$ in measure and $1-t,2t-1>0$, Lemmas~\ref{lem:trunc-powers} and \ref{lem:measure-cont}
		yield $f_t(a_n,b_n)\to f_t(a,b)$ in measure. Applying Lemma~\ref{lem:fatou-lsc} again gives
		\[
		\|f_t(a,b)\|_{\Etau}\le 2^{\max\{4t-3,0\}}\|a+b\|_{\Etau}.
		\]
		
		\smallskip
		\emph{Case 2: $\tau(\one)=\infty$ and $\one\notin E$.}
		Fix $k,n$ and apply the bounded-invertible estimate in the finite algebra $e_k\M e_k$ to obtain
		\[
		\|f_t(a_{k,n,\varepsilon},b_{k,n,\varepsilon})\|_{\Etau}
		\le 2^{\max\{4t-3,0\}}\ \|a_{k,n,\varepsilon}+b_{k,n,\varepsilon}\|_{\Etau}.
		\]
		Note that
		\[
		a_{k,n}+b_{k,n}=e_k\bigl((a\wedge n)+(b\wedge n)\bigr)e_k.
		\]
		Thus, by \eqref{eq:bimodule} and $(a\wedge n)+(b\wedge n)\le a+b$,
		\[
		\|a_{k,n}+b_{k,n}\|_{\Etau}
		\le \|(a\wedge n)+(b\wedge n)\|_{\Etau}
		\le \|a+b\|_{\Etau}.
		\]
		Moreover, $a_{k,n,\varepsilon}+b_{k,n,\varepsilon}=(a_{k,n}+b_{k,n})+2\varepsilon e_k$ and $\|e_k\|_{\Etau}<\infty$ by Remark~\ref{rem:finite-corner-in-Etau}.
		Hence
		\[
		\|a_{k,n,\varepsilon}+b_{k,n,\varepsilon}\|_{\Etau}\le \|a_{k,n}+b_{k,n}\|_{\Etau}+2\varepsilon\|e_k\|_{\Etau}.
		\]
		Letting $\varepsilon\downarrow0$ and using Lemma~\ref{lem:fatou-lsc} (for the left-hand side) yields
		\[
		\|f_t(a_{k,n},b_{k,n})\|_{\Etau}\le 2^{\max\{4t-3,0\}}\ \|a_{k,n}+b_{k,n}\|_{\Etau}
		\le 2^{\max\{4t-3,0\}}\ \|a+b\|_{\Etau}.
		\]
		For fixed $k$, letting $n\to\infty$ and using Lemma~\ref{lem:measure-cont} yields
		$f_t(a_{k,n},b_{k,n})\to f_t(e_kae_k,e_kbe_k)$ in measure, hence
		\[
		\|f_t(e_kae_k,e_kbe_k)\|_{\Etau}\le 2^{\max\{4t-3,0\}}\ \|a+b\|_{\Etau}\qquad(k\ge1).
		\]
		Finally, Lemma~\ref{lem:seq-cutdown} implies $e_kae_k\to a$ and $e_kbe_k\to b$ in measure, so again
		$f_t(e_kae_k,e_kbe_k)\to f_t(a,b)$ in measure, and Lemma~\ref{lem:fatou-lsc} yields the desired bound.
	\end{proof}
	
	\subsection{A bridge inequality via $2\times2$ block positivity}
	
	\begin{prop}\label{prop:bridge}
		Let $t\in[\tfrac12,1]$. Then
		\[
		\norm{b_t}_{\Etau}\le \norm{a+b}_{\Etau}^{1/2}\ \norm{f_t}_{\Etau}^{1/2}.
		\]
	\end{prop}
	
	\begin{proof}
		Let $t\in[\tfrac12,1]$. Consider the column operators
		\[
		u:=
		\begin{pmatrix}
			a^{1/2}\\[2pt]
			b^{1-t}a^{t-1/2}
		\end{pmatrix},
		\qquad
		v:=
		\begin{pmatrix}
			b^{1/2}\\[2pt]
			a^{1-t}b^{t-1/2}
		\end{pmatrix}.
		\]
		A direct computation in $\mathbb M_2(\Stau)$ gives
		\[
		uu^*=
		\begin{pmatrix}
			a & a^t b^{1-t}\\
			b^{1-t}a^t & b^{1-t}a^{2t-1}b^{1-t}
		\end{pmatrix},
		\qquad
		vv^*=
		\begin{pmatrix}
			b & b^t a^{1-t}\\
			a^{1-t}b^t & a^{1-t}b^{2t-1}a^{1-t}
		\end{pmatrix}.
		\]
		Adding,
		\[
		uu^*+vv^*
		=
		\begin{pmatrix}
			a+b & b_t\\
			b_t^* & f_t
		\end{pmatrix}
		\ge 0.
		\]
		By Lemma~\ref{lem:douglas}, there exists a contraction $w\in \M$ such that
		\[
		b_t=(a+b)^{1/2}\,w\,f_t^{1/2}.
		\]
		Since $\|w\|_\infty\le 1$, \eqref{eq:mu-axb} yields $\mu(w^*(a+b)w)\le \mu(a+b)$ and hence $w^*(a+b)w\precprec a+b$.
		By full symmetry, $\|w^*(a+b)w\|_{\Etau}\le \|a+b\|_{\Etau}$.
		
		Finally, apply Lemma~\ref{lem:holder} with $(2,2)$ to $x:=(a+b)^{1/2}w$ and $y:=f_t^{1/2}$:
		\[
		\|b_t\|_{\Etau}
		=\|(a+b)^{1/2}w f_t^{1/2}\|_{\Etau}
		\le \||x|^2\|_{\Etau}^{1/2}\ \||y|^2\|_{\Etau}^{1/2}
		= \|w^*(a+b)w\|_{\Etau}^{1/2}\ \|f_t\|_{\Etau}^{1/2}.
		\]
		Combining gives the claim.
	\end{proof}
	
	\subsection{Proof of Theorem~\ref{thm:main}}
	
	\begin{proof}[Proof of Theorem~\ref{thm:main}]
		Fix $t\in[\tfrac12,1]$.
		Proposition~\ref{prop:bridge} and Theorem~\ref{thm:f-bound} yield
\begin{align*}
	\norm{b_t}_{\Etau}
	&\le \norm{a+b}_{\Etau}^{1/2}\,\norm{f_t}_{\Etau}^{1/2}\\
	&\le \norm{a+b}_{\Etau}^{1/2}\,\bigl(2^{\max\{4t-3,0\}}\norm{a+b}_{\Etau}\bigr)^{1/2}\\
	&=2^{\max\{2t-\tfrac32,\,0\}}\norm{a+b}_{\Etau}.
\end{align*}

		For $t\in[0,\tfrac12]$, note that $b_t^*=b_{1-t}$ and $\|z\|_{\Etau}=\|z^*\|_{\Etau}$. Applying the above estimate to
		$1-t\in[\tfrac12,1]$ gives
		\[
		\norm{b_t}_{\Etau}=\norm{b_{1-t}}_{\Etau}
		\le 2^{\max\{2(1-t)-\tfrac32,\,0\}}\norm{a+b}_{\Etau}
		=2^{\max\{\tfrac12-2t,\,0\}}\norm{a+b}_{\Etau}.
		\]
		Combining the two ranges yields
		\[
		\norm{b_t}_{\Etau}\le 2^{\max\{\,2\abs{t-\tfrac12}-\tfrac12,\;0\,\}}\;\norm{a+b}_{\Etau},
		\]
		which is exactly \eqref{eq:main}. The ``in particular'' statement follows since the exponent
		vanishes for $t\in[\tfrac14,\tfrac34]$.
	\end{proof}
	
	\begin{rem}
		The factor $2^{\max\{\,2\abs{t-1/2}-1/2,\;0\,\}}$ equals $1$ on $[\tfrac14,\tfrac34]$ and
		reaches at most $\sqrt2$ on $[0,1]$.
		At the endpoints $t=0,1$, the expression involves the convention $a^0=s(a)$ and $b^0=s(b)$; in particular,
		if $a$ and $b$ have full support (e.g.\ are injective in a finite corner), then $b_0(a,b)=b_1(a,b)=a+b$
		and the optimal constant is $1$.
		
		\smallskip
		The constant $1$ in the central range $t\in[\tfrac14,\tfrac34]$ is sharp in the sense that it cannot be replaced by any
		$c<1$: indeed, for $a=b\neq0$ we have $b_t(a,a)=2a$ for all $t\in[0,1]$, while $a+b=2a$, hence equality holds in
		\eqref{eq:main}.
	\end{rem}
	
	\section*{Acknowledgments}
	Teng Zhang is supported by the China Scholarship Council, the Young Elite Scientists Sponsorship
	Program for PhD Students (China Association for Science and Technology), and the Fundamental
	Research Funds for the Central Universities at Xi'an Jiaotong University (Grant No.~xzy022024045).
	

\end{document}